\documentclass[makeidx]{article}
\usepackage[latin1]{inputenc}
\title{Sharp Steklov upper bound for submanifolds of revolution}
\author{Bruno Colbois and Sheela Verma}
\date{\today}
\usepackage{makeidx}
\usepackage{amssymb , amsmath, amsthm,graphicx, float}
\usepackage{graphicx}
\newtheorem{defi}{Definition}
\newtheorem{theorem}[defi]{Theorem}

 \newtheorem{prop}[defi]{Proposition}

\newcommand{\N}{\mathbb N}

\begin{document}
\maketitle
\large
\noindent

\abstract In this note, we find a sharp upper bound for the Steklov spectrum on a submanifold of revolution in Euclidean space with one boundary component.

\bigskip
\noindent {\it Classification AMS $2000$}: 35P15, 58C40\newline

\section{Introduction}
The Steklov eigenvalues of a smooth, compact, connected Riemannian
manifold $(M,g)$ of dimension $n\geq 2$ with boundary $\Sigma$ are the real
numbers $\sigma$ for which there exists a nonzero harmonic function
$f:M \rightarrow \mathbb R$ which satisfies $\partial_\nu f=\sigma f$ on the
boundary $\Sigma$. Here and in what follows,
$\partial_\nu$ is the outward-pointing normal derivative on $\Sigma$. The Steklov
eigenvalues form a discrete sequence
$0=\sigma_0<\sigma_1\leq\sigma_2\leq\cdots\nearrow\infty,$
where each eigenvalue is repeated according to its multiplicity.

Recently, relationships between the geometry of the boundary $\Sigma$ and the spectrum have been very much investigated. In \cite{CEG2}, the authors show that fixing only the geometry of the boundary and letting the Riemannian metric inside $M$ be arbitrary does not suffice to control the Steklov eigenvalues: they can be as large or as small as one wishes. On the other hand, it was shown in \cite{CGH, Xi} that fixing $g$ in a neighborhood of $\Sigma$ has a much stronger influence on the spectrum.

In \cite{CGG}, the authors consider $n$-dimensional submanifolds of revolution $M$ of the Euclidean space $\mathbb R^{n+1}$ with one boundary component $\mathbb S^{n-1}\subset \mathbb R^n\times \{0\}$. They show the sharp lower bound
$$
\sigma_k(M)\ge \sigma_k(\mathbb B^n), \quad \mbox{ for \ each } k \in \mathbb{N}\cup \{0\},
$$
where $\mathbb B^n$ denotes the submanifold of revolution given by the $n$ dimensional Euclidean ball. For $n\ge 3$, in the case of equality for one of the eigenvalues $\sigma_k$, $k\ge 1$, $M$ has to be isometric to $\mathbb B^n$. They also find an upper bound for $\sigma_k$, but it is not sharp. Note that, for $n=2$, all submanifold of revolution with boundary $S^1$ are isospectral (see \cite{CGG}, Proposition 1.10).

\medskip
The goal of this work is to investigate sharp upper bounds for submanifolds of revolution $M$ in Euclidean space $\mathbb R^{n+1}$ with one boundary component $\mathbb S^{n-1}\subset \mathbb R^n\times \{0\}$. We denote by
$$
0=\sigma_{(0)}(M)<\sigma_{(1)}(M) < \sigma_{(2)}(M)<...
$$
the \emph{distinct} (counted without multiplicity) eigenvalues of the submanifold of revolution $M$. Now we state our main result.

\begin{theorem} \label{main} Let $M \subset \mathbb R^{n+1}$ be an $n$-dimensional submanifold of revolution with one boundary component isometric to the round sphere $\mathbb S^{n-1}$. Then for $n \ge 3$, we have for each $k \ge 1$,
$$
\sigma_{(k)}(M) < k+n-2.
$$

Moreover, these bounds are sharp. For each $\epsilon >0$ and each $k\ge 1$, there exists a submanifold of revolution $M_{\epsilon}$ such that
$\sigma_{(k)} (M_{\epsilon}) > k+n-2-\epsilon.$

However, the inequality is strict: for each k, there does not exist submanifolds of revolution $M$ such that $\sigma_{(k)}(M)=k+n-2$.
\end{theorem}

\medskip
Note that such bounds exist for abstract revolution metrics on the ball $\mathbb B^n$ if we impose bounds on the curvature of $(M,g)$ (see \cite{Ve}, \cite{Xi1}). Roughly speaking, in \cite{Ve}, the author considers the Steklov problem on a ball with rotationally invariant metric under the assumption that the radial curvature is bounded below (or bounded above) by some real number and proves a two-sided bound for the Steklov eigenvalues. For warped product manifolds with only one boundary component, the author in \cite{Xi1} has obtained a lower bound (upper bound) for the Steklov eigenvalues under the hypothesis that the manifold has nonnegative (nonpositive) Ricci curvature and strictly convex boundary.

\medskip
Theorem \ref{main} will be a consequence of the study of the mixed Steklov Dirichlet and Steklov Neumann spectrum on an annulus and Proposition \ref{increase} which states that given a submanifold of revolution $M_1$ with one boundary component, it is always possible to construct another submanifold of revolution $M_2$ with larger Steklov eigenvalues.

\medskip
The rest of the paper is organized as follows. In Section \ref{section:general}, we present the Steklov and mixed Steklov problems. In Section \ref{revolution}, we consider the specific situation of submanifolds of revolution of Euclidean space with one boundary component. Finally, in Section \ref{proof}, we give the proof of Theorem \ref{main}.

\section{Some general facts about Steklov and mixed
   problems} \label{section:general}
Let $(M,g)$ be a compact Riemannian manifold with boundary $\Sigma$. The Steklov eigenvalues of $(M,g)$ can be characterized by the following variational formula
\begin{equation}
  \label{minmax}
  \sigma_j(M) = \min_{E \in \mathcal{H}_j} \max_{0 \neq f \in E} R_M(f), \quad j\ge 0,
\end{equation}
where $\mathcal{H}_j$ is the set of all $(j+1)$-dimensional subspaces in the Sobolev space $H^1(M)$ and
\begin{equation*}
  \label{Rayleighquotient}
  R_M(f) = \frac{\int_{M} \vert \nabla f \vert^2 dV_M}{\int_{\Sigma} \vert f \vert^2 dV_{\Sigma}}
\end{equation*}
is the Rayleigh quotient.

In order to obtain bounds for $\sigma_j(M)$, we will compare the Steklov spectrum with the spectra of the mixed Steklov-Dirichlet or Steklov-Neumann problems on domains $A \subset M$ such that $\Sigma \subset A$. Let $\partial_{int}A$ denote the intersection of the boundary of $A$ with the interior of $M$. Also, we suppose that it is smooth.

The mixed Steklov-Neumann problem on $A$ is the eigenvalue problem
\begin{gather*}
  \Delta f=0\mbox{ in } {A},\\
  \partial_\nu f=\sigma f\mbox{ on } \Sigma,\qquad
  \partial_\nu f=0 \mbox{ on }\ \partial_{int}A,
\end{gather*}
 where $\nu$ denotes the outward-pointing normal to $\partial A$.
The eigenvalues of this mixed problem form a discrete sequence
$$0=\sigma_0^N(A) \le \sigma_1^N( A)\leq\sigma_2^N(A)\leq\cdots\nearrow\infty,$$
and for each $j\geq 0$, the $j^{\text{th}}$ eigenvalue is given by
\begin{gather*}
 \sigma_j^{N}(A)=\min_{E\in\mathcal{H}_{j}(A)}\max_{0\neq f\in E}\frac{\int_{A}|\nabla f|^2\,dV_A}{\int_{\Sigma}|f|^2\,dV_{\Sigma}},
\end{gather*}
where $\mathcal{H}_{j}(A)$ is the set of all $(j+1)$-dimensional subspaces
in the Sobolev space $H^1(A)$.

\medskip
The mixed Steklov-Dirichlet problem
on $A$ is the eigenvalue problem
\begin{gather*}
  \Delta f=0\mbox{ in } {A},\\
  \partial_\nu f=\sigma f\mbox{ on } \Sigma,\qquad
  f=0 \mbox{ on }\ \partial_{int}A.
\end{gather*}
The eigenvalues of this mixed problem form a discrete sequence
$$0<\sigma_0^D( A)\leq\sigma_1^D(A)\leq\cdots\nearrow\infty,$$
and the $j^{\text{th}}$ eigenvalue is given by
\begin{gather*}
 \sigma_j^{D}(A)=\min_{E\in\mathcal{H}_{j,0}(A)}\max_{0\neq f\in E}\frac{\int_{A}|\nabla f|^2\,dV_A}{\int_{\Sigma}|f|^2\,dV_{\Sigma}},
\end{gather*}
where $\mathcal{H}_{j,0}(M)$ is the set of all $(j+1)$-dimensional subspaces
in the Sobolev space
$H_0^1(A)=\{u\in H^1(A)\,:\, u=0 \mbox{ on }\partial_{int}A\}.$

\medskip

For each $j\in\N$, comparisons between the variational formulae 
give the following bracketing:
\begin{gather}\label{ineq:compmixed}
\sigma_j^N(A) \le \sigma_j(M) \le \sigma_{j}^D(A).
\end{gather}
Note in particular that for $j=0$, we have
$$0=\sigma_0^N(A)=\sigma_0(M)<\sigma_0^D(A).$$

 \section{ Submanifolds of revolution of Euclidean space} \label{revolution} A compact submanifold of revolution $M^n$ with one boundary component is a revolution metric on the $n$-dimensional ball. It can be seen as the warped product $[0,L] \times S^{n-1}$ with the Riemannian metric
$$
g(r,p)=dr^2+h^2(p)g_{0}(p),
$$
where $(r,p)\in [0,L] \times S^{n-1}$, $g_0$ is the canonical metric on the sphere of radius one and $h \in C^{\infty}([0,L])$ satisfies $h>0$ on $[0,L[$, $h'(L)=1$ and $h^{2l}(L)=0$ for all integers $l\geq0$. If we suppose that the boundary is the round sphere of radius one, we also have $h(0)=1$. Moreover, the fact that $M$ is an $n$-dimensional submanifold of revolution of Euclidean space implies
$$
1-r\le h(r) \le 1+r.
$$
For more details, see \cite{CGG}.
\subsection{ Steklov spectrum and the eigenfunctions of submanifold of revolution}

The Steklov spectrum and the eigenfunctions of a submanifold of revolution with one connected component are very well explained in Proposition 8 of \cite{Xi1}. Before proceeding further, we would like to mention that by Laplace-Beltrami operator, we mean $\Delta=- \mbox{ div } \mbox{ grad}$, which is positive, whereas in \cite{Xi1}, the author considers $\Delta = \mbox{ div } \mbox{ grad}$. This explains the difference of the signs in the following.

\begin{prop} \label{prop: multiplicity of Steklov}
Let $M$ be a submanifold of revolution in Euclidean space $\mathbb R^{n+1}$ with one boundary component $\mathbb S^{n-1}\subset \mathbb R^n\times \{0\}$. Then each eigenfunction $f$ of the Steklov problem on $M$ can be written as $f(r,p)=u(r)v(p)$, where $v$ is a spherical harmonic of the sphere $\mathbb S^{n-1}$ of degree $k\ge 0$, i.e.,

\begin{align*}
\Delta v = \lambda_{(k)} v  \mbox { on } \mathbb S^{n-1},
\end{align*}
where $\lambda_{(k)}=k(n-2+k)$ and $u$ is a solution of the equation
\begin{align*}
\frac{1}{h^{n-1}} \frac{d}{dr}(h^{n-1}\frac{d}{dr}u)-\frac{1}{h^2}\lambda_{(k)} u=0
\end{align*}
on $(0,L)$ and under the condition $u(L)=0$. The Steklov eigenvalue $\sigma_{(k)}$ has the same multiplicity as of $\lambda_{(k)}$, the $k^{\text{th}}$ eigenvalue (counted without multiplicity) of the round sphere $\mathbb S^{n-1}$.
\end{prop}
For a proof and more details, see \cite{Xi1}.
\medskip

Roughly speaking, this comes from the fact that

\begin{align*}
\Delta (u(r)v(p))=-u''(r)v(p)-\frac{(n-1)h'}{h}u'(r)v(p)+\frac{u(r)}{h^2}\Delta_{\mathbb S^{n-1}} v(p).
\end{align*}
If $v$ is a $k^{\text{th}}$ eigenfunction of the sphere $\mathbb S^{n-1}$ (counted without multiplicity), we obtain
\begin{align*}
\Delta (u(r)v(p)) = -\frac{1}{h^{n-1}} \frac{d}{dr}(h^{n-1}\frac{d}{dr}u)v(p)+\frac{u(r)}{h^2}\lambda_{(k)} v(p),
\end{align*}
and, because $f$ is harmonic, we have
\begin{align*}
-\frac{1}{h^{n-1}} \frac{d}{dr}(h^{n-1}\frac{d}{dr}u)+ \lambda_{(k)}\frac{u(r)}{h^2} =0.
\end{align*}
The condition $u(L)=0$ comes from the fact that $f$ has to be smooth at the point where $r=L$.

Using similar arguments used in Proposition 8 of \cite{Xi1}, we can also conclude the following proposition.

\begin{prop} \label{prop: multiplicity of the mixed Steklov}
On a submanifold of revolution $M\subset \mathbb R^{n+1}$, the multiplicity of the $k^{\text{th}}$ mixed Steklov Dirichlet eigenvalue (counted without multiplicity) $\sigma^{D}_{(k)}$ and the $k^{\text{th}}$ mixed Steklov Neumann eigenvalue (counted without multiplicity) $\sigma^{N}_{(k)}$ has the same multiplicity as of $\lambda_{(k)}$, the $k^{\text{th}}$ eigenvalue (counted without multiplicity) of the round sphere $\mathbb S^{n-1}$.
\end{prop}

In Subsections \ref{SteklovDirichlet} and \ref{SteklovNeumann}, we calculate eigenvalues of the mixed Steklov problems on annular domains and these calculations are directly inspired from \cite{FS4}, paragraph 3.

\subsection{The mixed Steklov-Dirichlet eigenvalues on annular domains}\label{SteklovDirichlet}

\begin{prop}
Let $B_{1}$ and $B_{L}$ be the balls in $\mathbb{R}^{n}$, $n\ge 3$, centered at the origin of radius one and $L$, respectively. Consider the following eigenvalue problem on $\Omega_{0} = B_{L} \setminus \overline{B_{1}} $
\begin{align} \label{Steklov-Dirichlet}
\begin{array}{rcll}
\Delta f &=& 0 &\text{ in }  B_{L}\setminus \overline{B_{1}}, \\
f &=& 0  &\text{ on }  \partial B_{L}, \\
\frac{\partial f}{\partial \nu} &=& \sigma^{D} f  &\text{ on }  \partial B_{1}.
\end{array}
\end{align}
Then for $0 \leq k < \infty$, the $k^{\text{th}}$ eigenvalue (counted without multiplicity) of eigenvalue problem \eqref{Steklov-Dirichlet} is
\begin{align*}
\sigma^{D}_{(k)}(\Omega_{0}) = \frac{k}{L^{2k+n-2} - 1} + \frac{\left( k+n-2 \right) L^{2k+n-2} }{L^{2k+n-2} - 1}.
\end{align*}
\end{prop}

\begin{proof}
The eigenfunctions of \eqref{Steklov-Dirichlet} are of the form $f_k(r,p)=u(r)v(p)$, where $v$ is an eigenfunction for the $k^{\text{th}}$ eigenvalue of the sphere $\mathbb S^{n-1}$ and $u$ is a real-valued function defined on the interval $[1, L]$. For $f_k(r,p)$ to be an eigenfunction corresponding to the $k^{\text{th}}$ eigenvalue (counting without multiplicity) of the mixed Steklov Dirichlet problem on $\Omega_{0}$, $u$ should satisfy the following
\begin{align*}
\begin{array}{lcll}
u(r)= a r^{k} + b r^{-k+2-n}, \text{ for any nonnegative integer } k, \\
u(L) = 0, u'(1) = - \sigma^{D}_{(k)} u(1).
\end{array}
\end{align*}
Since $u'(r)=k ar^{k-1}-(n+k-2)br^{-(n+k-1)}$ for $k>0$, conditions $u(L) = 0$ and $u'(1) = - \sigma^{D}_{(k)} u(1)$ give
\begin{align*}
a L^{k} + b L^{-k+2-n} = 0, \\
k a + (-k+2-n) b = - \sigma^{D}_{(k)} (a + b).
\end{align*}
Eliminating $a$ and $b$ to obtain
\begin{align*}
L^{2k+n-2} (\sigma^{D}_{(k)} - k + 2 -n) = k + \sigma^{D}_{(k)}
\end{align*}
and
\begin{align*}
(L^{2k+n-2}-1) \sigma^{D}_{(k)}  = k + (n+k-2)L^{2k+n-2}.
\end{align*}
This gives the desired result.
\end{proof}

\subsection{The mixed Steklov-Neumann eigenvalues on annular domains}\label{SteklovNeumann}
\begin{prop}
Let $B_{1}$ and $B_{L}$ be the balls in $\mathbb{R}^{n}$, $n\ge 3$, centered at the origin of radius one and $L$, respectively. Consider the following eigenvalue problem on $\Omega_{0} = B_{L} \setminus \overline{B_{1}} $
\begin{align} \label{Steklov-Neumann}
\begin{array}{rcll}
\Delta f &=& 0 &\text{ in }  B_{L}\setminus \overline{B_{1}}, \\
\frac{\partial f}{\partial \nu}&=& 0  &\text{ on }  \partial B_{L}, \\
\frac{\partial f}{\partial \nu} &=& \sigma^{N} f  &\text{ on }  \partial B_{1}.
\end{array}
\end{align}
Then for $0 \leq k < \infty$, the $k^{\text{th}}$ eigenvalue (counted without multiplicity) of eigenvalue problem \eqref{Steklov-Neumann} is
\begin{align*}
\sigma^{N}_{(k)}(\Omega_{0}) = k \frac{(n+k-2)(L^{(n+2k-2)}-1)}{kL^{(n+2k-2)}+(n+k-2)}.
 \end{align*}
\end{prop}
\begin{proof}
Note that the eigenfunctions $f_k(r,p)$ of \eqref{Steklov-Neumann} can be expressed as $f_k(r,p)=u(r)v(p)$, where $v$ is an eigenfunction for the $k^{\text{th}}$ eigenvalue of the sphere $\mathbb S^{n-1}$ and $u$ is a real-valued function defined on $[1, L]$. If the function $u$ corresponds to the $k^{\text{th}}$ eigenvalue (counting without multiplicity) of the mixed Steklov-Neumann problem on $\Omega_{0}$, then
\begin{align*}
\begin{array}{lcll}
u(r)= a r^{k} + b r^{-k+2-n}, \text{ for any nonnegative integer } k, \\
u'(L) = 0, u'(1) = - \sigma^{N}_{(k)} u(1).
\end{array}
\end{align*}
These conditions give
\begin{align*}
a kL^{k-1} - b (n+k-2)L^{-(k+n-1)} = 0, \\
k a + (-k+2-n) b = - \sigma^{N}_{(k)} (a + b).
\end{align*}
By eliminating $a$ and $b$, we obtain
\begin{align*}
  -(k+\sigma^{N}_{(k)})(n+k-2)L^{-(k+n-1)}+kL^{k-1}(n+k-2-\sigma^{N}_{(k)})=0
\end{align*}
and
\begin{align*}
 \sigma^{N}_{(k)}(kL^{k-1}+(n+k-2)L^{-(k+n-1)})=k(n+k-2)(L^{k-1}-L^{-(k+n-1)}).
\end{align*}
Multiplying by $L^{(n+k-1)}$ to get
\begin{align*}
\sigma^{N}_{(k)}(kL^{(n+2k-2)}+(n+k-2))=k(n+k-2)(L^{(n+2k-2)}-1),
\end{align*}
and
\begin{align*}
\sigma^{N}_{(k)}=k \frac{(n+k-2)(L^{(n+2k-2)}-1)}{kL^{(n+2k-2)}+(n+k-2)}.
\end{align*}
\end{proof}

 \section{Proof of the main theorem} \label{proof}
\subsection{Comparison of submanifolds of revolution}

Recall that for an $n$-dimensional submanifold of revolution $M$ of the Euclidean space $\mathbb R^{n+1}$ with one boundary component $\mathbb S^{n-1}\subset \mathbb R^n\times \{0\}$ the induced Riemannian metric may be written as
$$
g(r,p)=dr^2+h^2(r)g_0(p),
$$
where $g_0$ is the canonical metric of $\mathbb S^{n-1}$, $r \in [0,L]$ and $h(0)=1$, $h(L)=0$, $h(r)>0$ if $0<r<L$, $h'(L)=1$ and $1-r\le h(r) \le 1+r$.\\

\begin{prop} \label{comparison}
Let $M=[0,L]\times \mathbb{S}^{n-1}$, $n \ge 3$ be a Riemannian manifold with metric $g_{i} = dr^{2} + h_{i}^{2}(r) g_{\mathbb{S}^{n-1}}$, $i = 1, 2$. Moreover suppose that $h_{1}(0) = h_{2}(0)=1$ and $h_{1}(r) \le h_{2}(r)$. Consider the mixed Steklov-Neumann problem on $M$ (Steklov at $r = 0$ and Neumann at $r = L$). Then  $\sigma_{k}^{N}(M, g_{1}) \le \sigma_{k}^{N}(M, g_{2})$ for all $k \in \mathbb{N}\cup \{0\}$.
\end{prop}

\begin{proof}
The Rayleigh quotient of a function $f(r,p)$ defined on $M$ is given by
\begin{align*}
R_{g_{i}}(f) = \frac{\int_{0}^{L} \int_{\mathbb{S}^{n-1}} \left( \left( \frac{\partial f}{\partial r}\right)^{2} + \frac{1}{h_{i}^{2}(r)} \|\bar{\nabla} f \|^{2}\right)h_{i}^{n-1}(r) dr \ dv_{g_{\mathbb{S}^{n-1}}}}{\int_{\mathbb{S}^{n-1} \times \{0\}}f^{2}(0,p) dv_{g_{\mathbb{S}^{n-1}}}},
\end{align*}
where $\bar{\nabla}$ is the exterior derivative in the direction of $\mathbb{S}^{n-1}$. The condition $h_{1}(r) \le h_{2}(r)$ gives that $R_{g_{1}}(f) \le R_{g_{2}}(f)$. Hence, we have $\sigma_{k}^{N}(M, g_{1}) \le \sigma_{k}^{N}(M, g_{2})$ for all $k \in \mathbb{N}\cup \{0\}$.
\end{proof}
\vspace{0.3 cm}

\begin{prop} \label{increase}
For any submanifold of revolution $(M_{1}, g_{1}) \subset \mathbb{R}^{n+1}$, with boundary $\mathbb{S}^{n-1} \times \left\lbrace 0 \right\rbrace $, there exists a submanifold of revolution $(M_{2}, g_{2}) \subset \mathbb{R}^{n+1}$ with the same boundary such that, for all $k \ge 1$, $\sigma_{(k)}(M_{2}) > \sigma_{(k)}(M_{1})$.
\end{prop}
\begin{proof}
Note that $M_{1}$ will be of the form $[0,L_{1}]\times \mathbb{S}^{n-1}$ with metric $g_{1} = dr^{2} + h_{1}^{2}(r) g_{\mathbb{S}^{n-1}}$, where $h_{1}$ satisfies $h_{1}(0)=1$, $\vert h_{1}'(r) \vert \leq 1$ and $h_{1}(L_{1})=0$. The condition $\vert h_{1}'(r) \vert \leq 1$ gives $ 1-r \leq h_{1}(r) \leq 1+r $.

\noindent Consider a submanifold of revolution $M_{2} = [0,L_{2}] \times \mathbb{S}^{n-1}$ with metric $g_{2} = dr^{2} + h_{2}^{2}(r) g_{\mathbb{S}^{n-1}}$, where $L_{2}= 2 L_{1} + 2$ and
\begin{align*}
h_{2}(r)=
\begin{cases}
1+r, & \text{ if } r \leq L_{1}, \\
L_{2} - r, & \text{ if } L_{1} +1\leq r \leq L_{2}.
\end{cases}
\end{align*}
For $L_1\le r \le L_1+1$, we just ask that $h$ joins $h(L_1)$ and $h(L_1+1)$ smoothly. Note that $h_{1}(r) \le h_{2}(r)$ for $r \le L_{1}$.

\noindent Now consider the mixed Steklov-Neumann problem on $\tilde{M} = [0,L_{1}- \epsilon] \times \mathbb{S}^{n-1}$ with two metrics $g_{1}$ and $g_{2}$. Then from Proposition \ref{prop: multiplicity of the mixed Steklov} and \ref{comparison}, we get $\sigma_{(k)}^{N}(\tilde{M}, g_{1}) < \sigma_{(k)}^{N}(\tilde{M}, g_{2})$ for all $k\ge 1 $. The strict inequality follows from Proposition \ref{comparison} applied to eigenfunctions of $(\tilde{M}, g_{1})$ and from the fact that there exist points in $\tilde{M}$ such that $h_2(r)>h_1(r)$ at those points.

\noindent Recall that because of the bracketing,
$$\sigma_{(k)}(M_{2}, g_{2}) \geq \sigma_{(k)}^{N}(\tilde{M}, g_{2}), \quad k \in \mathbb{N}\cup \{0\}.$$
Using the method of Anné (see \cite{An1}, Theorem 2, and \cite{An2} for a less general but easiest version of the result), we have that as $\epsilon \rightarrow 0$, $\sigma_{(k)}^{N}(\tilde{M}, g_{1}) \rightarrow \sigma_{(k)}(M_{1}, g_{1})$. As a consequence, we get $\sigma_{(k)}(M_{2}, g_{2}) > \sigma_{(k)}(M_{1}, g_{1}) $.
\end{proof}

Next we prove Theorem \ref{main} by using Proposition \ref{increase}.
\begin{proof}
Note that $M$ will be of the form $[0,L] \times \mathbb{S}^{n-1}$ with metric $g = dr^{2} + h^{2}(r) g_{\mathbb{S}^{n-1}}$, where $h$ satisfies $h(0)=1$, $\vert h'(r) \vert \leq 1$ and $h(L)=0$.

\noindent Proposition \ref{increase} already shows that it is always possible to strictly increase the spectrum of $M$. Moreover, Proposition \ref{increase} gives the existence of a sequence of submanifolds of revolution $M_{i} = [0,L_{i}] \times \mathbb{S}^{n-1}$, $1 \leq i < \infty$, with boundary $\mathbb{S}^{n-1} \times \{ 0\} $ and metric $g_{i} = dr^{2} + h_{i}^{2}(r) g_{\mathbb{S}^{n-1}}$ ($h_{i}$ and $L_{i}$ are constructed as in Proposition \ref{increase}) such that
\begin{align*}
\sigma_{(k)}(M) < \sigma_{(k)}(M_{1}) < \sigma_{(k)}(M_{2}) < \cdots.
\end{align*}
Also, for $i \geq 2$,
\begin{align*}
\sigma^{N}_{(k)}(A_{i}) \leq \sigma_{(k)}(M_{i}) \leq \sigma^{D}_{(k)}(A_{i}),
\end{align*}
where $A_{i}$ is an annular domain with inner radius one and outer radius $1+ L_{i-1}$, and it is a neighborhood of the boundary of $M_{i}$.

\noindent Moreover, we have $L_i \to \infty$ as $i\to \infty$. Note that for $k>0$,
\begin{align*}
\lim_{i \rightarrow \infty} \sigma^{D}_{(k)}(A_{i}) = \lim_{i \rightarrow \infty} \sigma^{N}_{(k)}(A_{i}) = k+n-2.
\end{align*}
This shows $\lim_{i \rightarrow \infty} \sigma_{(k)}(M_{i}) = k+n-2$. Combining this with the fact that $\sigma_{(k)}(M_{i})$ is an increasing sequence proves the theorem.
\end{proof}

\subsection*{Acknowledgment}
The authors sincerely thank Dr. Katie Gittins for her valuable comments and suggestions.

\addcontentsline{toc}{chapter}{Bibliography}
\bibliographystyle{plain}
\bibliography{biblioCV}

\bigskip
\normalsize
\noindent Bruno Colbois \\
Universit\'e de Neuch\^atel, Institut de Math\'ematiques \\
Rue Emile-Argand 11\\
 CH-2000, Neuch\^atel, Suisse

\noindent bruno.colbois@unine.ch

\medskip

\normalsize
\noindent
Sheela Verma\\
Tata Institute of Fundamental Research\\
Centre For Applicable Mathematics\\
Bangalore, India

\noindent sheela.verma23@gmail.com
\end{document}